\documentclass[reqno,11pt]{amsart}
\usepackage[foot]{amsaddr}
\usepackage{bbold}
\usepackage{charter}
\usepackage[margin=1in]{geometry}
\usepackage[colorlinks=true,linktoc=all,linkcolor=blue,citecolor=blue]{hyperref}

\theoremstyle{plain}
\newtheorem{theorem}{Theorem}[section]
\newtheorem*{theorem*}{Theorem}
\newtheorem{proposition}[theorem]{Proposition}
\newtheorem{lemma}[theorem]{Lemma}
\newtheorem{corollary}[theorem]{Corollary}
\theoremstyle{definition}

\newtheorem*{definition*}{Definition}
\newtheorem{remark}[theorem]{Remark}

\newtheorem{observation}[theorem]{Observation}
\numberwithin{equation}{section}
\everymath{\displaystyle}
\allowdisplaybreaks

\newcommand{\ds}{\displaystyle}
\newcommand{\C}{\mathbb{C}}
\newcommand{\D}{\mathbb{D}}

\newcommand{\N}{\mathbb{N}}

\newcommand{\Bloch}{\mathcal{B}}
\newcommand{\lilBloch}{\mathcal{B}_0}
\newcommand{\Aut}{\mathrm{Aut}}

\newcommand{\ord}{\mathrm{ord}}
\newcommand{\Log}{\mathrm{Log}}
\newcommand{\id}{\mathrm{id}}

\newcommand{\qte}[1]{``#1"}
\newcommand{\defeq}{\stackrel{\mathrm{def}}{=}}
\newcommand{\ip}[1]{\langle #1 \rangle}
\newcommand{\norm}[1]{\left|\left|#1\right|\right|}
\newcommand{\supnorm}[1]{\norm{#1}_\infty}
\newcommand{\blochnorm}[1]{\norm{#1}_\Bloch}
\renewcommand{\mod}[1]{\left|#1\right|}
\newcommand{\conj}[1]{\overline{#1}}
\newcommand{\mulop}[1]{M_{#1}}

\newcommand{\wcompop}[1]{W_{#1}}
\newcommand{\mop}{\mulop{\psi}}

\newcommand{\wco}{\wcompop{\psi,\varphi}}

\title[Isometric Multiplication Operators]{Isometries and Spectra of Multiplication Operators on the Bloch Space}

\author{Robert F. Allen and Flavia Colonna}

\address{Department of Mathematical Sciences, George Mason University}
\email{rallen2@gmu.edu, fcolonna@gmu.edu}

\date{}


\begin{document}
\begin{abstract} In this paper, we establish bounds on the norm of multiplication operators on the Bloch space of the unit disk via weighted composition operators.  In doing so, we characterize the isometric multiplication operators to be precisely those induced by constant functions of modulus 1.  We then describe the spectrum of the multiplication operators in terms of the range of the symbol.  Lastly, we identify the isometries and spectra of a particular class of weighted composition operators on the Bloch space.
\end{abstract}

\maketitle

\section{Introduction}

Let $\D$ denote the open unit disk in the complex plane.  An analytic function $f$ on $\D$ is said to be \emph{Bloch} if  \[\beta_f = \sup_{z \in \D}\;(1-\mod{z}^2)\mod{f'(z)} < \infty.\]  The mapping $f \mapsto \beta_f$ is a semi-norm on the space $\Bloch$ of Bloch functions, called the \emph{Bloch space}.  Under the norm \[\blochnorm{f} = \mod{f(0)} + \beta_f,\] the Bloch space is a Banach space. The \emph{little Bloch space} $\lilBloch$ is the closed subspace of $\Bloch$ consisting of functions $f \in \Bloch$ satisfying \[\lim_{\mod{z} \to 1} (1-\mod{z}^2)\mod{f'(z)} = 0.\]  For further treatment of Bloch functions and the Bloch space see \cite{Pommerenke:70}, \cite{AndersonCluniePomerenke:74}, and \cite{Cima:79}.

Suppose $\psi$ is a fixed analytic function on $\D$.  The \emph{multiplication operator} on the Bloch space is defined as \[\mop(f) = \psi f.\]  In \cite{BrownShields:91}, Brown and Shields proved that $M_\psi$ is a bounded operator on the Bloch space if and only if $\psi \in H^\infty(\D)$ and \[\mod{\psi'(z)} = O\left(\frac{1}{(1-\mod{z})\log\frac{1}{1-\mod{z}}}\right).\]  This implies that \[(1-\mod{z}^2)\mod{\psi'(z)} = O\left(\frac{1}{\log\frac{1}{1-\mod{z}}}\right) \to 0\] as  $\mod{z} \to 1$.  Thus $\psi \in \lilBloch$.

Suppose further that $\varphi$ is a fixed analytic self-map of $\D$.  The \emph{weighted composition operator} on the Bloch space is defined as \[\wco(f) = \psi(f\circ\varphi).\]  There is a reason for introducing the weighted composition operator at this point.  The multiplication operator is a \qte{degenerate} weighted composition operator; if $\varphi$ is the identity map, then $\wco = \mop$.  Likewise, the \emph{composition operator} on the Bloch space, defined as \[C_\varphi(f) = f\circ\varphi,\] is a \qte{degenerate} weighted composition operator when $\psi$ is taken to be the constant function 1.

In \cite{OhnoZhao:01}, Ohno and Zhao proved that $\wco$ is a bounded operator on the Bloch space if and only if the following two properties are satisfied:

\begin{equation}
\label{wco condition 1} \ds \sup_{z \in \D}\; (1-\mod{z}^2)\mod{\psi'(z)}\log\frac{2}{1-\mod{\varphi(z)}^2} < \infty,
\end{equation}
\begin{equation}
\label{wco condition 2} \ds \sup_{z \in \D}\; \frac{1-\mod{z}^2}{1-\mod{\varphi(z)}^2}\mod{\psi(z)}\mod{\varphi'(z)} < \infty.
\end{equation}  It can easily be verified that if $\wco$ is bounded on the Bloch space, then $\psi$ is a Bloch function.

The characterization of the isometries is an open problem for most Banach spaces of analytic functions; the set of isometries, in most cases, is too large to study all at once.  Even so, there are spaces for which the isometries are known.  In \cite{Forelli:64}, Forelli proved the isometries on the Hardy space $H^p$, for $p \neq 2$ are certain weighted composition operators.  Also, El-Gebeily and Wolfe completely characterized the isometries on the disk algebra $\mathcal{A}_0 = H^\infty(\D) \cap C(\partial\D)$ \cite{ElGebeilyWolfe:85}.  For most other spaces, a complete picture of the isometries is not known.  The surjective isometries for the Bergman space $A^p$ (see \cite{Rudin:76}) and the weighted Bergman space $A^p_\omega$ (see \cite{Kolaski:82}) are known.

The first attempt at studying the isometries on the Bloch space was made by Cima and Wogen in \cite{CimaWogen:80}.  They studied the isometries on the subspace of the Bloch space whose elements fix the origin.  On this space, they showed that the surjective isometries are normalized compressions of composition operators induced by disk automorphisms.  However, on the entire set of Bloch functions, a description of all isometries is still unknown.  The current trend in the literature is to attack this problem one operator at a time.

To date, the only complete characterization of isometries on the Bloch space are those among the composition operators.  Independently, two different descriptions have been developed.  In \cite{Colonna:05}, the second author identified the isometric composition operators in terms of the canonical factorization of the symbol.  Mart\'{i}n and Vukoti\'{c} in \cite{MartinVukotic:07}, identified the isometric composition operators in terms of the hyperbolic derivative and cluster set of the symbol.  The purpose of this paper is to further develop the understanding of the isometries on the Bloch space by characterizing the isometric multiplication operators.  The authors are unaware of any such description of the isometric multiplication operators on the Bloch space in the literature.

\subsection{Organization of the paper}
In Section 2, we determine estimates on the norm of the weighted composition operator $\wco$ on the Bloch space.  As a corollary, we deduce estimates on the norm of the multiplication operator $\mop$.  In the case where $\psi \equiv 1$, the norm estimates agree with those derived by Xiong in \cite{Xiong:04} for the composition operator $C_\varphi$.

In Section 3, we prove that the symbols of isometric multiplication operators are precisely the constant functions of modulus 1.

In Section 4, we show that the spectrum of the multiplication operator is the closure of the range of the symbol.  From this, we deduce the spectra of the isometric multiplication operators to be single-element subsets of the unit circle.

In Section 5, we discuss the weighted composition operators whose symbols induce isometries individually and describe their spectrum.

\section{Norm Estimates on Weighted Composition Operators}
In this section we establish estimates on $\norm{\wco}$.  In order to write these estimates in a succinct way, we introduce the following notation.  Let $\psi \in \Bloch$ and $\varphi:\D \to \D$ analytic functions satisfying conditions (\ref{wco condition 1}) and (\ref{wco condition 2}).  Let
\[\begin{aligned}
\tau_{\psi,\varphi}^\infty &\defeq \sup_{z \in \D}\;\frac{1-\mod{z}^2}{1-\mod{\varphi(z)}^2}\mod{\psi(z)}\mod{\varphi'(z)}, \text{ and}\\
\sigma_{\psi,\varphi}^\infty &\defeq \sup_{z \in \D}\;\frac{1}{2}(1-\mod{z}^2)\mod{\psi'(z)}\log\frac{1+\mod{\varphi(z)}}{1-\mod{\varphi(z)}},
\end{aligned}\] which are both finite.  In \cite{Xiong:04}, Xiong defined \[\tau_\varphi^\infty = \sup_{z \in \D}\frac{1-\mod{z}^2}{1-\mod{\varphi(z)}^2}\mod{\varphi'(z)}\] to obtain an upper estimate on the norm of a composition operator on $\Bloch$.  Note, if $\psi \equiv 1$, then $\tau_{\psi,\varphi}^\infty = \tau_\varphi^\infty$.  Also, if $\varphi$ is the identity map, then
\[\sigma_{\psi,\varphi}^\infty = \sigma_\psi^\infty \defeq \sup_{z \in \D}\;\frac{1}{2}(1-\mod{z}^2)\mod{\psi'(z)}\log\frac{1+\mod{z}}{1-\mod{z}}.\]

To determine an upper bound on $\norm{\wco}$ we use the following lemma, a proof of which can be found in \cite{Colonna:88} or \cite{Xiong:04}.

\begin{lemma}\label{inequality lemma} If $f \in \Bloch$, then \[\label{xiong_identity} \mod{f(z)} \leq \mod{f(0)} + \frac{1}{2}\beta_f\log\frac{1+\mod{z}}{1-\mod{z}},\] for all $z \in \D$.
\end{lemma}

\noindent With this lemma and the notation defined above, we are now prepared to establish an upper bound on $\norm{\wco}$.

\begin{theorem}Suppose $\psi$ is an analytic function on $\D$ and $\varphi$ is an analytic self-map of $\D$ that induce a bounded weighted composition operator $\wco$ on $\Bloch$.  Then
\[\norm{\wco} \leq \max\left\{\blochnorm{\psi}, \frac{1}{2}\mod{\psi(0)}\log\frac{1+\mod{\varphi(0)}}{1-\mod{\varphi(0)}} + \tau_{\psi,\varphi}^\infty + \sigma_{\psi,\varphi}^\infty\right\}.\]
\end{theorem}

\begin{proof}
Let $f \in \Bloch$ such that $\blochnorm{f} = 1$.  Then
\[\begin{aligned}
\blochnorm{\wco f} &\leq \mod{\psi(0)}\mod{f(\varphi(0))} + \sup_{z \in \D}\;(1-\mod{z}^2)\mod{\psi(z)}\mod{f'(\varphi(z))}\mod{\varphi'(z)}\\
    &\qquad + \sup_{z \in \D}\;(1-\mod{z}^2)\mod{\psi'(z)}\mod{f(\varphi(z))}\\
&= \mod{\psi(0)}\mod{f(\varphi(0))} + \sup_{z \in \D}\;\frac{1-\mod{z}^2}{1-\mod{\varphi(z)}^2}\mod{\psi(z)}\mod{\varphi'(z)}(1-\mod{\varphi(z)}^2)\mod{f'(\varphi(z))}\\
    &\qquad + \sup_{z \in \D}\;(1-\mod{z}^2)\mod{\psi'(z)}\mod{f(\varphi(z))}\\
&\leq \mod{\psi(0)}\mod{f(\varphi(0))} + \tau_{\psi,\varphi}^\infty\beta_f + \sup_{z \in \D}\;(1-\mod{z}^2)\mod{\psi'(z)}\mod{f(\varphi(z))}.
\end{aligned}\]
By Lemma \ref{inequality lemma}, we have $\mod{f(\varphi(z))} \leq \mod{f(0)} + \frac{1}{2}\beta_f\log\ds\frac{1+\mod{\varphi(z)}}{1-\mod{\varphi(z)}}$.  Thus
\[\begin{aligned}
\blochnorm{\wco f} &\leq \mod{\psi(0)}\mod{f(\varphi(0))} + \tau_{\psi,\varphi}^\infty\beta_f + \beta_\psi\mod{f(0)} + \sigma_{\psi,\varphi}^\infty\beta_f.
\end{aligned}\]
Since $\mod{f(\varphi(0))} \leq \mod{f(0)} + \frac{1}{2}\beta_f\log\ds\frac{1+\mod{\varphi(0)}}{1-\mod{\varphi(0)}}$, and recalling that $\mod{f(0)} = 1-\beta_f$ we deduce
\[\begin{aligned}
\blochnorm{\wco f} &\leq \blochnorm{\psi}\mod{f(0)} + \left(\frac{1}{2}\mod{\psi(0)}\log\frac{1+\mod{\varphi(0)}}{1-\mod{\varphi(0)}} + \tau_{\psi,\varphi}^\infty + \sigma_{\psi,\varphi}^\infty\right)\beta_f\\
&= \blochnorm{\psi} + \left(\frac{1}{2}\mod{\psi(0)}\log\frac{1+\mod{\varphi(0)}}{1-\mod{\varphi(0)}} + \tau_{\psi,\varphi}^\infty + \sigma_{\psi,\varphi}^\infty - \blochnorm{\psi}\right)\beta_f.
\end{aligned}\]

\noindent If $\displaystyle\frac{1}{2}\mod{\psi(0)}\log\frac{1+\mod{\varphi(0)}}{1-\mod{\varphi(0)}} + \tau_{\psi,\varphi}^\infty + \sigma_{\psi,\varphi}^\infty \leq \blochnorm{\psi}$, then \[\blochnorm{\wco f} \leq \blochnorm{\psi}.\]  If $\displaystyle\frac{1}{2}\mod{\psi(0)}\log\frac{1+\mod{\varphi(0)}}{1-\mod{\varphi(0)}} + \tau_{\psi,\varphi}^\infty + \sigma_{\psi,\varphi}^\infty \geq \blochnorm{\psi}$, then
\[\begin{aligned}
\blochnorm{\wco f} &\leq \blochnorm{\psi} + \frac{1}{2}\mod{\psi(0)}\log\frac{1+\mod{\varphi(0)}}{1-\mod{\varphi(0)}} + \tau_{\psi,\varphi}^\infty + \sigma_{\psi,\varphi}^\infty - \blochnorm{\psi}\\
&= \frac{1}{2}\mod{\psi(0)}\log\frac{1+\mod{\varphi(0)}}{1-\mod{\varphi(0)}} + \tau_{\psi,\varphi}^\infty + \sigma_{\psi,\varphi}^\infty.
\end{aligned}\]
Therefore, \[\norm{\wco} \leq \max\left\{\blochnorm{\psi}, \frac{1}{2}\mod{\psi(0)}\log\frac{1+\mod{\varphi(0)}}{1-\mod{\varphi(0)}} + \tau_{\psi,\varphi}^\infty + \sigma_{\psi,\varphi}^\infty\right\},\] as desired.
\end{proof}

To determine a lower bound on $\norm{\wco}$, we apply the appropriate test functions.

\begin{theorem} Suppose $\psi$ is an analytic function on $\D$ and $\varphi$ is an analytic self-map of $\D$ inducing a bounded weighted composition operator $\wco$ on $\Bloch$.  Then
\begin{equation}\label{wco lower bound}\norm{\wco} \geq \max\left\{\blochnorm{\psi}, \frac{1}{2}\mod{\psi(0)}\log\frac{1+\mod{\varphi(0)}}{1-\mod{\varphi(0)}}\right\}.\end{equation}
\end{theorem}

\begin{proof} If we take the test function $f(z) = 1$, then
\[\norm{\wco} \geq \blochnorm{\wco f} = \blochnorm{\psi}.\]

\noindent If $\varphi(0) = 0$, then (\ref{wco lower bound}) holds trivially.  If $\varphi(0) \neq 0$, then write $\varphi(0) = \mod{\varphi(0)}e^{i\theta}$, for some $\theta \in [0,2\pi)$.  Let $f(z) = \displaystyle\frac{1}{2}\Log\frac{1+e^{-i\theta}z}{1-e^{-i\theta}z}$, where $\Log$ denotes the principal branch of the logarithm.  Then $f$ is a Bloch function with Bloch norm 1.  Thus
\[\norm{\wco} \geq \blochnorm{\wco f} \geq \mod{\psi(0)f(\varphi(0))} = \frac{1}{2}\mod{\psi(0)}\log\frac{1+\mod{\varphi(0)}}{1-\mod{\varphi(0)}}.\]

\noindent Therefore, \[\norm{\wco} \geq \max\left\{\blochnorm{\psi}, \frac{1}{2}\mod{\psi(0)}\log\frac{1+\mod{\varphi(0)}}{1-\mod{\varphi(0)}}\right\},\] as desired.
\end{proof}

\begin{remark} By taking $\psi \equiv 1$, we deduce the estimates on the norm of $C_\varphi$ established in \cite{Xiong:04}:

\[\max\left\{1, \frac{1}{2}\log\frac{1+\mod{\varphi(0)}}{1-\mod{\varphi(0)}}\right\} \leq \norm{C_\varphi} \leq \max\left\{1, \frac{1}{2}\log\frac{1+\mod{\varphi(0)}}{1-\mod{\varphi(0)}} + \tau_\varphi^\infty\right\}.\]
\end{remark}

We now deduce estimates on the norm of the multiplication operator $M_\psi$ on $\Bloch$.

\begin{corollary}\label{bounds mult op} Suppose $\psi$ is an analytic function on $\D$ inducing a bounded multiplication operator $M_\psi$ on $\Bloch$.  Then
\[\max\left\{\blochnorm{\psi},\supnorm{\psi}\right\} \leq \norm{M_\psi} \leq \max\left\{\blochnorm{\psi}, \supnorm{\psi} + \sigma_\psi^\infty\right\}.\]  In particular, if $\psi(0) = 0$, then \[\supnorm{\psi} \leq \norm{\mop} \leq \supnorm{\psi}+\sigma_\psi^\infty.\]
\end{corollary}

\begin{proof}  By taking $\varphi$ to be the identity map, $\tau_{\psi,\varphi}^\infty = \supnorm{\psi}$ and $\sigma_{\psi,\varphi}^\infty = \sigma_\psi^\infty$.  Thus,
\[\blochnorm{\psi} \leq \norm{M_\psi} \leq \max\left\{\blochnorm{\psi}, \supnorm{\psi}+\sigma_\psi^\infty\right\}.\]
Furthermore, we have $\supnorm{\psi} \leq \norm{\mop}$.  Indeed, this is true for a bounded multiplication operator on any \emph{functional Banach space}, see \cite{DurenRombergShields:69} Lemma 11.  Therefore, \[\max\left\{\blochnorm{\psi},\supnorm{\psi}\right\} \leq \norm{\mop} \leq \max\left\{\blochnorm{\psi}, \supnorm{\psi}+\sigma_\psi^\infty\right\},\] as desired.  The conclusion for the case $\psi(0) = 0$ follows from the fact that by the Schwarz-Pick lemma, for a bounded function $\psi$ on $\D$, \[(1-\mod{z}^2)\frac{\mod{\psi'(z)}}{\supnorm{\psi}} \leq 1-\frac{\mod{\psi(z)}^2}{\supnorm{\psi}^2}\] for all $z$ in $\D$, so that $\blochnorm{\psi} = \beta_\psi \leq \supnorm{\psi}$.
\end{proof}

\noindent\textbf{Open Question. }  If $\mop$ is bounded on $\Bloch$, is $\blochnorm{\psi} \leq \supnorm{\psi} + \sigma_\psi^\infty$?  As mentioned above, this is true if $\psi$ fixes the origin.  The inequality also holds in the case of $\psi$ being an automorphism of $\D$.  Indeed, suppose $\psi \in \Aut(\D)$ with $\psi(0) \neq 0$.  Then \[\psi(z) = \eta \frac{a-z}{1-\conj{a}z},\] where $0 \neq a \in \D$ and $\mod{\eta} = 1$.  Thus, $\supnorm{\psi} = 1$ and $\beta_\psi = (1-\mod{a}^2)\mod{\psi'(a)} = 1$, so $\blochnorm{\psi} = \mod{a} + 1$.  Furthermore
\[\begin{aligned}
\sigma_\psi^\infty &= \sup_{z \in \D}\;(1-\mod{z}^2)\mod{\psi'(z)}\frac{1}{2}\log\frac{1+\mod{z}}{1-\mod{z}}\\
&\geq (1-\mod{a}^2)\mod{\psi'(a)}\frac{1}{2}\log\frac{1+\mod{a}}{1-\mod{a}}\\
&= \frac{1}{2}\log\frac{1+\mod{a}}{1-\mod{a}} > \mod{a}
\end{aligned}\] since, for $\mod{a} < 1$, \[\frac{1}{2}\log\frac{1+\mod{a}}{1-\mod{a}} = \sum_{n=0}^\infty \frac{\mod{a}^{2n+1}}{2n+1}.\]  Therefore $\blochnorm{\psi} \leq \supnorm{\psi} + \sigma_\psi^\infty$.

\section{Characterization of Isometric Multiplication Operators}
This section is devoted to the identification of the symbols of the isometric multiplication operators on the Bloch space.  We first establish necessary conditions for $\psi$ to induce an isometric multiplication operator on $\Bloch$.

\begin{lemma}\label{super lemma} Let $\psi$ be the symbol of an isometric multiplication operator on $\Bloch$.  Then $\supnorm{\psi} \leq 1$ and $\blochnorm{\psi^n} = 1$ for all $n \in \N$.\end{lemma}

\begin{proof} By the lower estimate in Corollary \ref{bounds mult op} we obtain $\supnorm{\psi} \leq \norm{M_\psi} = 1.$  By choosing the test function $g \equiv 1$, we have \[1 = \blochnorm{g} = \blochnorm{\mop g} = \blochnorm{\psi}.\]  Thus
\[\blochnorm{\psi^2} = \blochnorm{M_\psi(\psi)} = \blochnorm{\psi} = 1.\]  The conclusion follows by induction.\end{proof}

The next result is more general, and with it, we can prove the main theorem.

\begin{lemma}\label{power norm lemma} If $\psi \in H^\infty(\D)$ such that $\supnorm{\psi} \leq 1$ and $\psi(0) = 0$, then $\blochnorm{\psi^n} < 1$ for all $n \geq 2$.\end{lemma}

\begin{proof}
By the Schwarz-Pick lemma, for all $n \in \N$, $n \geq 2$, we have
\[\begin{aligned}
\beta_{\psi^n} &= \sup_{z \in \D}\;(1-\mod{z}^2)n\mod{\psi(z)}^{n-1}\mod{\psi'(z)}\\
&\leq \sup_{z \in \D}\;n(1-\mod{\psi(z)}^2)\mod{\psi(z)}^{n-1}\\
&\leq n\max_{x \in [0,1]} (x^{n-1}-x^{n+1})\\
&= \frac{2n}{n+1}\left(\frac{n-1}{n+1}\right)^{\frac{n-1}{2}} = \frac{2n}{n-1}\left(\frac{n-1}{n+1}\right)^{\frac{n+1}{2}}.
\end{aligned}\]  For $x$ and $a$ positive real numbers and $m$ a real number greater than 1, we have $(x+a)^m > x^m + amx^{m-1}$.  By this, we have for $n \geq 2$,
\[(n+1)^{\frac{n+1}{2}} > (n-1)^{\frac{n+1}{2}} + 2\left(\frac{n+1}{2}\right)(n-1)^{\frac{n-1}{2}} = 2n(n-1)^{\frac{n-1}{2}}.\]  So \[\frac{2n}{n+1}\left(\frac{n-1}{n+1}\right)^{\frac{n-1}{2}} = \frac{2n(n-1)^{\frac{n-1}{2}}}{(n+1)^{\frac{n+1}{2}}} < 1.\]  Hence $\blochnorm{\psi^n} = \beta_{\psi^n} < 1$ for $n \geq 2$.
\end{proof}

\begin{corollary}\label{fix origin corollary} If $\psi$ is the symbol of an isometric multiplication operator on $\Bloch$, then $\psi$ does not fix the origin.\end{corollary}

\begin{proof} Arguing by contradiction, assume $\psi(0) = 0$.  By Lemma \ref{super lemma}, $\supnorm{\psi} \leq 1$ and $\blochnorm{\psi^2} = 1$.  However, $\blochnorm{\psi^2} < 1$ by Lemma \ref{power norm lemma}.
\end{proof}

\begin{lemma}\label{zeros lemma} Suppose $\psi \in H^\infty(\D)$ such that $\supnorm{\psi} \leq 1$ and the map $g(z) = z\psi(z)$ has Bloch norm 1.  Then either $\psi$ is a constant of modulus 1, or $\psi$ has infinitely many zeros $\{a_n\}$ in $\D$ such that \[\beta_\psi = \limsup_{n \to \infty}\; (1-\mod{a_n}^2)\mod{\psi'(a_n)} = 1.\]  In the latter case, if $\blochnorm{\psi} = 1$, then $\psi(0) = 0$.\end{lemma}

\begin{proof}  Assume $\psi$ is not a constant of modulus 1.  Note that the function $g$ maps $\D$ into itself, fixes the origin, and has Bloch norm 1.  Then by Theorem 3 of \cite{Colonna:05}, there exists an infinite sequence $\{a_n\}$ in $\D$ such that $\psi(a_n) = 0$ and \[\limsup_{n \to \infty}\;(1-\mod{a_n}^2)\mod{g'(a_n)} = 1.\]  Since $g'(a_n) = \psi(a_n) + a_n\psi'(a_n) = a_n\psi'(a_n)$, we deduce \[\limsup_{n \to \infty}\;(1-\mod{a_n}^2)\mod{a_n}\mod{\psi'(a_n)} = 1.\]  By assumption $\psi$ is not a constant function, and so $\mod{a_n} \to 1$ as $n \to \infty$.   Thus  \[\beta_\psi = \limsup_{n \to \infty}\;(1-\mod{a_n}^2)\mod{\psi'(a_n)} = 1.\]  The conclusion for the case $\blochnorm{\psi} = 1$ follows at once.
\end{proof}

We now prove the main theorem of this section.

\begin{theorem}\label{characterization theorem} The multiplication operator $M_\psi$ is an isometry on $\Bloch$ if and only if $\psi$ is a constant function of modulus 1.\end{theorem}

\begin{proof}  Clearly, if $\psi$ is a constant function of modulus 1 then $M_\psi$ is an isometry on $\Bloch$.  Conversely, suppose $M_\psi$ is an isometry on $\Bloch$ and assume $\psi$ is not a constant function of modulus 1.  Then by Lemma \ref{super lemma}, $\supnorm{\psi} \leq 1$ and $\blochnorm{\psi} = 1$.  Also, for $g(z) = z\psi(z)$, $\blochnorm{g} = \blochnorm{M_\psi(\id)} = \blochnorm{\id} = 1$, where $\id$ is the identity map on $\D$.  Then by Lemma \ref{zeros lemma}, $\psi(0) = 0$, contradicting Corollary \ref{fix origin corollary}.  Therefore, if $M_\psi$ is an isometry on $\Bloch$, then $\psi$ must be a constant function of modulus 1.
\end{proof}

\section{Characterization of the Spectra of Multiplication Operators}
We now turn our attention to the spectrum of a multiplication operator.  Recall that the \emph{resolvent} of a bounded linear operator $T$ on a complex
Banach space $E$ is defined as
\[\rho(T) = \{\lambda \in \C : T-\lambda I \text{ is invertible}\},\] where $I$ is the identity operator.  The \emph{spectrum} of $T$ is defined as $\sigma(T) =
\C\setminus\rho(T)$.  The \emph{approximate point spectrum}, a
subset of the spectrum, is defined as
\[\sigma_{ap}(T) = \{\lambda \in \C : T-\lambda I \text{ is not
bounded below}\},\] i.e. for every $M>0$, there exists $x \in E$ such that $\norm{Tx} < M\norm{x}$.

The spectrum is a non-empty compact subset of the closed disk centered at the origin of radius $\norm{T}$.  In particular, the spectrum of an isometry is contained in $\overline{\D}$.  Furthermore the boundary $\partial\sigma(T)$ of $\sigma(T)$ is a subset of $\sigma_{ap}(T)$ (see \cite{Conway:90}, Proposition~6.7).

\begin{theorem}\label{spectrum M_psi} Let $\psi$ be the symbol of a bounded multiplication operator $\mop$ on $\Bloch$.  Then $\sigma(\mop) = \overline{\psi(\D)}$.\end{theorem}

\begin{proof}  For $\lambda \in \C$, the operator $M_\psi - \lambda I$ can be rewritten as $M_{\psi - \lambda}$.  Thus $\lambda \in \sigma(M_\psi)$ if and only if $M_{\psi-\lambda}$ is not invertible.  Clearly, if $M_{\psi-\lambda}^{-1}$ exists, it is the multiplication operator $M_{(\psi-\lambda)^{-1}}$.

Let $\lambda \in \psi(\D)$.  Then there exists $z_0 \in \D$ such that $\psi(z_0) = \lambda$.  So $(\psi-\lambda)^{-1}$ has a pole at $z_0$, which means $M_{(\psi-\lambda)^{-1}}$ is not a well-defined operator.  Thus $M_{\psi-\lambda}$ is not invertible.  This implies that $\overline{\psi(\D)} \subseteq \sigma(M_\psi)$.

Suppose $\lambda \not\in \overline{\psi(\D)}$.  Then $\mod{\psi-\lambda}$ is bounded away from 0 by some positive constant $c$.  Thus the modulus of the function $g(z) = \frac{1}{\psi(z)-\lambda}$ is in $H^\infty(\D)$.  In addition, \[\mod{g'(z)} = \frac{\mod{\psi'(z)}}{\mod{\psi(z)-\lambda}^2} \leq \frac{1}{c^2} \mod{\psi'(z)} = O\left(\frac{1}{(1-\mod{z})\log\frac{1}{1-\mod{z}}}\right).\]  So $M_g = M_{(\psi-\lambda)^{-1}}$ is a bounded operator on $\Bloch$.  Thus $\lambda \not\in \sigma(M_\psi)$.  Therefore $\sigma(M_\psi) = \overline{\psi(\D)}$.
\end{proof}

This result is not surprising since it also holds for the space of continuous, real-valued functions on a closed interval.  A similar result holds for $L^2(\mu)$, $\mu$ a probability measure.  Specifically, for $\psi \in L^\infty(\mu)$, the spectrum of $M_\psi$ on $L^2(\mu)$ is the \emph{essential range} of $\psi$, that is the set of $\lambda \in \C$ such that the preimage under $\psi$ of every neighborhood of $\lambda$ has positive measure \cite{Douglas:98}.  As an immediate consequence of Theorems \ref{characterization theorem} and \ref{spectrum M_psi}, we obtain the following result.

\begin{corollary}\label{iso spectrum characterization corollary} Let $M_\psi$ be an isometric multiplication operator on the Bloch space.  Then $\sigma(M_\psi) = \{\eta\},$ where $\eta$ is the unimodular constant value of $\psi$.\end{corollary}

\section{A Further Glimpse at Weighted Composition Operators}
Our focus returns to weighted composition operators, and isometries amongst them.  Let $\mathcal{I}_W$ be the set of weighted composition operators $\wco$ such that $\psi$ induces an isometric multiplication operator and $\varphi$ induces an isometric composition operator.  Clearly the set of isometric weighted composition operators contains $\mathcal{I}_W$.

\begin{observation}\label{iso comp op obs} In Theorem 2 of \cite{Colonna:05}, it was shown that $\varphi$ induces an isometric composition operator on $\Bloch$ if and only if $\varphi(0) = 0$ and $\beta_\varphi = 1$.  In particular, Corollary 2 of \cite{Colonna:05} proved that $\varphi$ is either a rotation or the zero set of $\varphi$ forms an infinite sequence $\{a_n\}$ in $\D$ such that \[\limsup_{n \to \infty}\;(1-\mod{a_n}^2)\mod{\varphi'(a_n)} = 1.\]  Notice that the class $\mathfrak{I}$ of such functions $\varphi$ is very large.  Indeed, $\varphi \in \mathfrak{I}$ if and only if $\varphi = gB$ where $g$ is a non-vanishing analytic function of $\D$ into $\overline{\D}$, and $B$ is a Blaschke product whose zeros form an infinite sequence $\{a_n\}$ containing 0 and an infinite subsequence $\{a_{n_j}\}$ such that $\mod{g(a_{n_j})} \to 1$ and \begin{equation}\label{blaschke iso property}\lim_{j \to \infty} \prod_{k \neq n_j} \mod{\frac{a_{n_j}-a_k}{1-\conj{a_{n_j}}a_k}} = 1.\end{equation}  The Blaschke products whose zeros satisfy (\ref{blaschke iso property}) include the \emph{thin} Blaschke products.  For more on this topic, see \cite{Hedenmalm:87} and \cite{GorkinMortini:04}.\end{observation}

\noindent \textbf{Open Question. }  Are there any isometric weighted composition operators whose symbols do not induce isometric multiplication and composition operators individually?

If the answer to the above question is negative, then $\mathcal{I}_W$ is exactly the set of isometric weighted composition operators on the Bloch space and we have a complete characterization of their symbols.

At this time, we describe the spectrum of isometric composition operators, which with Corollary \ref{iso spectrum characterization corollary} will be used to describe the spectrum of the elements of $\mathcal{I}_W$.  To do so, we need an important result about the spectrum of an isometry on a general Banach space.  The following proposition is found (in a slightly different form) as an exercise in \cite{Conway:90}, and we provide a proof for completeness.

\begin{proposition}\label{conway proposition} Let $E$ be a complex Banach space and suppose $T:E \to E$ is an isometry.  If $T$ is invertible, then $\sigma(T) \subseteq \partial\D$.  If $T$ is not invertible, then $\sigma(T) = \overline{\D}$.
\end{proposition}

\begin{proof} Suppose $T$ is an invertible isometry on $E$.  Then $0 \not\in \sigma(T)$, and so the function $z \mapsto z^{-1}$ is analytic in some neighborhood of $\sigma(T)$.  By the Spectral Mapping
Theorem (see \cite{Conway:90}, p. 204), we have $\sigma(f\circ T) = f(\sigma(T))$, and so \[\sigma(T^{-1}) = \sigma(T)^{-1} = \{\lambda^{-1} : \lambda \in \sigma(T)\}.\]  Since $T^{-1}$ exists and is an isometry, we have $\sigma(T^{-1}) \subseteq \overline{\D}$.  Therefore $\sigma(T)
\subseteq \partial\D$.

Next, suppose $T$ is not invertible.  In order to prove that $\sigma(T) = \overline{\D}$, it suffices to show that $\D \subseteq \sigma(T)$.  For $\lambda \in \D$, $T-\lambda I$ is bounded below by $1-\mod{\lambda}$.  Thus, $\lambda \not\in \sigma_{ap}(T)$.  We deduce that
$\partial\sigma(T)\subseteq \sigma_{ap}(T) \subseteq \partial\D$.

Assume $\lambda \in \overline{\D}\cap\rho(T)$.  Note that $\lambda \not\in \partial\sigma(T)$ since $\partial\sigma(T) = \sigma(T) \cap
\overline{\rho(T)}$.  Consider $\Gamma = \{t\lambda : t \in [0,\infty)\}$, the radial line through $\lambda$.  Since $\sigma(T)$ is closed and $0 \in \sigma(T)$, there exists $t \in [0,1)$ such that $t\lambda \in \partial\sigma(T)$.  This contradicts the fact that $\partial\sigma(T) \subseteq \partial\D$.  Consequently, $\D \subseteq \sigma(T)$.\end{proof}

For a complex number $\zeta$ of modulus 1, define the \emph{order of $\zeta$}, denoted by $\ord(\zeta)$, to be the smallest $n \in \N$ such that $\zeta^n = 1$.  If no such $n$ exists, we say $\zeta$ has infinite order, and write $\ord(\zeta) = \infty.$

\begin{theorem}\label{spectrum iso C_phi theorem} Suppose $\varphi$ induces an isometric composition operator on $\Bloch$.  If $\varphi$ is not a rotation, then $\sigma(C_\varphi) = \overline{\D}$.  If $\varphi(z) = \zeta z$ with $\mod{\zeta} = 1$, then \[\sigma(C_\varphi) = \begin{cases}\partial\D & \text{if } \ord(\zeta) = \infty,\\\ip{\zeta}& \text{if } \ord(\zeta) < \infty,\end{cases}\] where $\ip{\zeta}$ is the cyclic group generated by $\zeta$.\end{theorem}

\begin{proof}
Assume $\varphi$ is not a rotation.  Then by Proposition \ref{conway proposition} it suffices to show that $C_\varphi$ is not invertible.  Since $C_\varphi$ is an isometry it is necessarily injective.  Thus, we will show $C_\varphi$ is not surjective.  Arguing by contradiction, assume $C_\varphi$ is surjective.  Since $\varphi$ is not a rotation, by Observation \ref{iso comp op obs} $\varphi$ has infinitely many zeros.  Let $a$ and $a'$ be distinct zeros of $\varphi$.

Define $h(z) = z - a$, and note that $h$ is a non-zero Bloch function.  Since $C_\varphi$ is assumed to be surjective, there exists a Bloch function $f$ such that $h = f\circ\varphi$.  On the one hand, $f(0) = f(\varphi(a')) = h(a') \neq 0$.  On the other hand, $f(0) = f(\varphi(a)) = h(a) = 0$, a contradiction.  Thus, $C_\varphi$ is not surjective.  Therefore $\sigma(C_\varphi) = \overline{\D}$.

Now suppose $\varphi(z) = \zeta z$ with $\mod{\zeta} = 1$.  Then $\varphi^{-1}(z) = \frac{1}{\zeta}z$ and $C_\varphi^{-1} = C_{\varphi^{-1}}$.  So by Proposition \ref{conway proposition}, $\sigma(C_\varphi) \subseteq \partial\D$.

Let $G = \ip{\zeta} = \left\{\zeta^k : k \in \N\cup\{0\}\right\}$.  Note that $G \subseteq \partial\D$.  Consider the Bloch function $f(z) = z^k$ for $k \in \N\cup\{0\}$.  Then \[(C_\varphi f)(z) = \zeta^kz^k = \zeta^k f(z).\]  Thus $\zeta^k$ is an eigenvalue of $C_\varphi$ with corresponding eigenfunction $f$.  So $G \subseteq \sigma(C_\varphi)$.

If the order of $\zeta$ is infinite, then $G$ is dense in $\partial\D$.  Since the spectrum is closed, we have $\partial\D = \overline{G} \subseteq \sigma(C_\varphi)$.  Thus $\sigma(C_\varphi) = \partial\D$.

Now suppose $\ord(\zeta) = n < \infty$.  Then $G = \{\zeta^k : k = 1, \dots, n\}$.  We wish to show that $\sigma(C_\varphi) \subseteq G$.  Let $\mu \in \partial\D\setminus G$.  We will show that $C_\varphi - \mu I$ is invertible by proving that for every $g \in \Bloch$, there exists a unique $f \in \Bloch$ such that $f\circ\varphi - \mu f = g.$  Since $\ord(\zeta) = n$, then $\varphi^{(n)}(z) \defeq (\underbrace{\varphi\circ\cdots\circ\varphi}_{\text{$n$-times}})(z) = \zeta^n z = z$.  By repeated application of $\varphi$, we can form the following system of equations:
\begin{equation}\label{spec-sys}
\begin{array}{lclcl}
f(\varphi(z)) & - & \mu f(z) & = & g(z) \\
f(\varphi^{(2)}(z)) & - & \mu f(\varphi(z)) & = & g(\varphi(z)) \\
& \vdots &  & $\vdots$ &  \\
f(z) & - & \mu f(\varphi^{(n-1)}(z)) & = &
g(\varphi^{(n-1)}(z)).
\end{array}
\end{equation}  Equivalently, (\ref{spec-sys}) can be posed as the matrix equation
$Ax = b$ where
\[A = \left[\begin{matrix}
    -\mu & 1 & 0 & 0 & \cdots & 0\\
    0 & -\mu & 1 & 0 & \cdots & 0\\
    \vdots & 0 & \ddots & \ddots & & \vdots\\
    \vdots & & \ddots & \ddots & \ddots & \vdots\\
    0 & & & \ddots & \ddots & 1\\
    1 & 0 & \cdots & \cdots & 0 & -\mu
    \end{matrix}\right], x = \left[\begin{matrix}
    f(z)\\f(\varphi(z))\\
    \vdots\\
    \vdots\\f(\varphi^{(n-2)}(z))\\f(\varphi^{(n-1)}(z))
    \end{matrix}\right], b =
    \left[\begin{matrix}
    g(z)\\g(\varphi(z))\\
    \vdots\\
    \vdots\\g(\varphi^{(n-2)}(z))\\g(\varphi^{(n-1)}(z))
    \end{matrix}\right].\]
Direct calculation shows that $\mathrm{det}(A) = (-1)^n(\mu^n-1)\ne 0$, since $\mu\notin G$. Thus,
$C_\varphi - \mu I$ is invertible. For $\mu \notin G$, the unique solution $f$ of (\ref{spec-sys}) is a (finite) linear combination of the functions $g\circ\varphi^{(j-1)}$, for $j=1,\dots,n$, each of which is Bloch. Thus $f$ is Bloch. Therefore $\sigma(C_\varphi) = G$.\end{proof}

\begin{theorem} Let $\psi$ induce an isometric multiplication operator (so that $\psi$ is a constant $\eta$ of modulus 1) and $\varphi$ induce an isometric composition operator.  If $\varphi$ is not a rotation, then $\sigma(\wco) = \overline{\D}$.  If $\varphi = \zeta z$ for $\mod{\zeta} = 1$, then \[\sigma(\wco) = \begin{cases}\partial\D& \text{if } \ord(\zeta) = \infty,\\\eta\ip{\zeta}& \text{if } \ord(\zeta) < \infty,\end{cases}\] where $\eta\ip{\zeta} = \{\eta\zeta^k : k = 1,\dots,n\}$.\end{theorem}

\begin{proof}
Observe that $\wco = \eta C_\varphi$ and, for $\lambda \in \C$, $\wco - \lambda I = \eta(C_\varphi - \lambda\conj{\eta}I)$.  Thus, $\wco - \lambda I$ is not invertible if and only if $C_\varphi - \lambda\conj{\eta}I$ is not invertible.  Thus $\lambda \in \sigma(\wco)$ if and only if $\lambda\conj{\eta} \in \sigma(C_\varphi)$.  The result follows immediately from Theorem \ref{spectrum iso C_phi theorem}.
\end{proof}

\bibliographystyle{amsplain}
\bibliography{references.bib}
\end{document}